\documentclass[reqno]{amsart}
\usepackage{amsmath,amsthm,amscd,amssymb,amsfonts, amsbsy}
\usepackage{latexsym, color, enumerate}
\usepackage{pxfonts}
\usepackage[mathscr]{eucal}

\theoremstyle{plain}
\newtheorem{theorem}[equation]{Theorem}
\newtheorem{lemma}[equation]{Lemma}
\newtheorem{corollary}[equation]{Corollary}

\theoremstyle{definition}

\newtheorem{assumption}[equation]{Assumption}

\theoremstyle{remark}
\newtheorem{remark}[equation]{Remark}

\numberwithin{equation}{section}

\newcommand{\bR}{\mathbb{R}}

\def\bH{\mathbb{H}}

\newcommand\cH{\mathcal{H}}

\newcommand\cX{\mathcal{X}}

\newcommand\Div{\operatorname{div}}
\newcommand{\wei}[1]{\langle #1 \rangle}

\providecommand{\set}[1]{\{#1\}}

\providecommand{\norm}[1]{\lVert#1\rVert}

\begin{document}
\title[Stokes systems with VMO coefficients]{Mixed-norm $L_p$-estimates for non-stationary Stokes systems with singular VMO coefficients and applications}

\author[H. Dong and T. Phan]{Hongjie Dong and Tuoc Phan}
\address[H. Dong]{Division of Applied Mathematics, Brown University,
182 George Street, Providence, RI 02912, USA}
\email{Hongjie\_Dong@brown.edu}
\address[T. Phan]{Department of Mathematics, University of Tennessee,  227 Ayres Hall, 1403 Circle Drive, Knoxville TN 37996, USA}
\email{phan@math.utk.edu}
\thanks{H. Dong was partially supported by the NSF under agreement  DMS-1600593; T. Phan is partially supported by the Simons Foundation, grant \# 354889.}
\subjclass[2010]{Primary: 76D03, 76D05, 76D07; Secondary: 35K67, 35K40}
\keywords{time-dependent Stokes system, mixed-norm regularity estimates, Navier-Stokes equations, Leray-Hopf weak solutions, regularity criteria}
\begin{abstract}
We prove the mixed-norm Sobolev estimates for solutions to both divergence and non-divergence form time-dependent Stokes systems with unbounded measurable coefficients having small mean oscillations with respect to the spatial variable in small cylinders. As a special case, our results imply Caccioppoli's type estimates for the Stokes systems with variable coefficients.  A new $\epsilon$-regularity criterion for Leray-Hopf weak solutions of Navier-Stokes equations is also obtained as a consequence of our regularity results, which in turn implies some borderline cases of the well-known Serrin's regularity criterion.
\end{abstract}
\maketitle

\today

\section{Introduction and main results}

In this paper, we study mixed-norm estimates in Sobolev spaces for solutions of non-stationary Stokes systems with measurable singular coefficients in both divergence and non-divergence forms.  Due to the singularity of the coefficients, our established results will be used to prove a new $\epsilon$-regularity criterion for Leray-Hopf weak solutions of the Navier-Stokes equations. Precisely, we study the following time-dependent Stokes system with general coefficients:
\begin{equation}
                        \label{eq7.41b}
u_t-D_{i}(a_{ij}D_j u)+\nabla p=\Div f,\quad \Div u= g,
\end{equation}
where $u = u(t,x) \in \bR^d$ is an unknown vector solutions representing the velocity of the considered fluid, $p = p(t,x)$ is an unknown fluid pressure. Moreover,  $f(t,x)=(f_{ij}(t,x))$ is a given measurable matrix of external forces, $g = g(t,x)$ is a given measurable function, and $a_{ij}=b_{ij}(t,x) + d_{ij}(t,x)$ is a given measurable matrix of viscosity coefficients that satisfies the following boundedness and ellipticity conditions with ellipticity constant $\nu \in (0,1)$:
\begin{equation} \label{ellipticity}
\nu |\xi|^2 \leq a_{ij} \xi_i \xi_j, \quad |b_{ij}| \leq \nu^{-1},
\end{equation}
and
\begin{equation} \label{symmetry}
 b_{ij} = b_{ji}, \quad d_{ij} \in L_{1,\text{loc}}, \quad d_{ij} = - d_{ji}, \quad \forall \ i, j \in \{1, 2,\ldots, d\}.
\end{equation}
Our goal is to establish regularity estimates for the gradient $Du$ of weak solutions $u$ of \eqref{eq7.41b} in the Lebesgue mixed-norm $L_{s,q}$.

We also consider non-divergence form Stokes systems
\begin{equation}
                        \label{eq7.41c}
u_t-a_{ij}D_{ij} u+\nabla p=f,\quad \Div u= g,
\end{equation}
and in this non-divergence form setting, the matrix $a_{ij} = b_{ij}$, i.e., $d_{ij} =0$,  $f = (f_1, f_2, \ldots, f_d)$ is given measurable vector field function, and $g = g(t,x)$ is a given measurable function. Regularity estimates of $D_{ij}u$ in Lebesgue mixed-norm $L_{s,q}$ will be established for strong solutions $u$ of \eqref{eq7.41c}.

The interest in results concerning equations in spaces with mixed Sobolev norms arises, for example, when one wants to have better regularity of traces of solutions for each time slide while treating linear or nonlinear equations. See, for instance, \cite{M-Sol, Sol04}, where the initial-boundary value problem for the non-stationary Stokes system in mixed-norm Sobolev spaces was studied.
Besides its mathematical interests, our motivation to study the Stokes systems \eqref{eq7.41b} and \eqref{eq7.41c} with variable coefficients comes from the study of inhomogeneous fluid with density dependent viscosity, see \cite{AGZ, Lions}, as well as the study of the Navier-Stokes equations in general Riemannian manifolds, see \cite{DM}. Moreover,  such problem is also connected to the study of regularity for weak solutions of the Navier-Stokes equations as we will explain in Theorem \ref{NS-reg.thm} and Corollary \ref{coro} below.

In Theorem \ref{thm2.3} below, we establish mixed-norm Sobolev estimate for gradients of weak solutions of \eqref{eq7.41b}. Meanwhile,  Theorem \ref{thm2.3b} below is about mixed-norm Sobolev estimates for second spatial derivatives of solutions of \eqref{eq7.41c}. In Theorem \ref{NS-reg.thm} we give a new $\epsilon$-regularity criterion for  Leray-Hopf weak solutions of the Navier-Stokes equations. Observe that in \eqref{ellipticity}, the boundedness of coefficients is only required for the symmetric part of the coefficients matrix $a_{ij}$. Therefore, the coefficient $a_{ij}$ can be singular, and this is the key point in Theorem \ref{NS-reg.thm}, which is an application of Theorem \ref{thm2.3} to the Navier-Stokes equations.

Before we state these results precisely, we introduce some notation and assumptions that we use in this paper. In addition the ellipticity condition \eqref{ellipticity}, we need the following VMO$_x$ (vanishing mean oscillation in $x$) condition, first introduced in \cite{Kry07}, with constants $\delta \in (0,1)$ and $\alpha_0 \in [1, \infty)$ to be determined later.
\begin{assumption}[$\delta, \alpha_0$]
                        \label{assump1}
There exists $R_0\in (0, 1/4)$ such that for any $(t_0,x_0)\in
Q_{2/3}$ and $r\in (0,R_0)$, there exists $\bar a_{ij}(t) = \bar{b}_{ij}(t) + \bar{d}_{ij}(t)$ for which $\bar{b}_{ij}(t)$ and $\bar{d}_{ij}(t)$ satisfy \eqref{ellipticity}-\eqref{symmetry} and
$$
\fint_{Q_r(t_0,x_0)} |a_{ij}(t,x)-\bar a_{ij}(t)|^{\alpha_0}\,dx\,dt\le \delta^{\alpha_0},
$$
where $\delta \in (0,1)$ and $\alpha_0 \in [ 1,\infty)$.
\end{assumption}
We note that in  the Assumption ($\delta, \alpha_0$) above, $Q_\rho(z_0)$ denotes the parabolic cylinder centered at $z_0 = (t_0, x_0) \in \bR^{d+1}$ with radius $\rho>0$. Precisely,
\[
Q_\rho(z_0) = (t_0 - \rho^2, t_0] \times B_\rho(x_0),
\]
where $B_\rho(x_0)$ denotes the ball in $\mathbb{R}^d$ of radius $\rho$ centered at $x_0 \in \mathbb{R}^d$.  For abbreviation, when $z_0 = (0,0)$, we write $Q_\rho = Q_\rho(0,0)$ and $B_\rho = B_\rho(0)$.

For each $s, q \in [1, \infty)$ and each parabolic cylinder $Q  = \Gamma \times U \subset \bR \times \bR^{d}$, the mixed $(s,q)$-norm of a function $u$ defined in $Q$ is
\[
\norm{u}_{L_{s,q}(Q)} = \left[\int_{\Gamma} \left( \int_{U} |u(t,x)|^q \,dx \right)^{s/q} dt \right]^{1/s}.
\]
As usual, we denote
\[
L_{s,q}(Q)= \{u : Q \rightarrow \mathbb{R}: \|u\|_{L_{s,q}(Q)} <\infty\} \quad \text{and} \quad L_{q}(Q) = L_{q,q}(Q).
\]
We also denote the parabolic Sobolev space
\begin{align*}
W_{s,q}^{1,2}(Q)&=
\set{u:\,u, Du,D^2u\in L_{s,q}(Q), \ u_t \in L_{1}(Q)},
\end{align*}
and denote $\bH^{-1}_{s,q}(Q)$ the space consisting of all functions $u$ satisfying
$$
\set{u=\Div F +h\ \text{in} \ Q: \|F\|_{L_{s,q}(Q)}+\|h\|_{L_{s,q}(Q)} <\infty}.
$$
Naturally, for any $u\in \bH^{-1}_{s,q}(Q)$, we define the norm
\begin{equation*}
\|u\|_{\bH^{-1}_{s,q}(Q)}=\inf\set{\|F\|_{L_{s,q}(Q)}+\|h\|_{L_{s,q}(Q)}\,|\,u=\Div F +h},
\end{equation*}
and it is easy to see that $\bH^{-1}_{s,q}(Q)$ is a Banach space. Moreover, when $u \in \bH^{-1}_{s,q}(Q)$ and $u = \Div F +h$, we write
\[
\wei{u, \phi} = \int_{Q} \Big[ -F \cdot \nabla \phi + h \phi \Big] \,dx\, dt, \quad \text{for any} \quad \phi \in C_0^\infty(Q).
\]
We also define
$$
\cH^{1}_{s,q}(Q)=
\set{u:\,u,Du \in L_{s,q}(Q),u_t\in \bH^{-1}_{1,1}(Q)}.
$$
When $s= q$, we will omit one of these two indices and write
\[
L_q(Q) = L_{q,q}(Q), \quad W_{q}^{1,2}(Q) = W_{q,q}^{1,2}(Q), \quad \cH^{1}_{q}(Q) = \cH^{1}_{q,q}(Q), \quad \bH^{-1}_{q}(Q) = \bH^{-1}_{q,q}(Q).
\]

\begin{remark} Note that in our definition of $W_{s,q}^{1,2}(Q)$ we only require that $u_t \in L_{1}(Q)$, not $ u_t \in L_{s,q}(Q)$ as in the standard notation for the space $W_{s,q}^{1,2}(Q)$. Similarly in the definition of $\cH^{1}_{s,q}(Q)$, we only require $u_t \in \bH^{-1}_{1,1}(Q)$, not $u_t \in \bH^{-1}_{s,q}(Q)$. This is because for the Stokes systems, local weak solutions may not possess good regularity in the time variable in view of Serrin's example \cite{Serrin}. It is possible to further relax the regularity assumptions of $u$ in $t$ and also $p$ below, but we do not pursue in that direction.
\end{remark}

For $s, q \in (1, \infty)$, we denote $s', q'$ the conjugates of $s,q$, i.e.,
\begin{equation} \label{conjugate}
1/s + 1/s' =1, \quad 1/q + 1/q' =1
\end{equation}
Then, under the assumption that $d_{kj} \in L_{s', q'}(Q_1)$, $f_{kj} \in L_{1}(Q_1)$ for all $i, j = 1,2,\ldots, d$, $g \in L_1(Q_1)$, and the ellipticity assumption \eqref{ellipticity}, we say that a vector field function $u = (u_1, u_2, \ldots, u_d) \in \cH^{1}_{s,q}(Q_1)^d$ is a weak solution of the Stokes system \eqref{eq7.41b} in $Q_1$ if
\[
\int_{B_1} u (t, x) \cdot \nabla \varphi(x)\,  dx  = -\int_{B_1} g(t,x) \varphi(x)\, dx, \ \text{for a.e.} \,\, t \in (-1,0), \ \text{for all} \,\, \varphi \in C_0^\infty(B_1)
\]
and
\[
\wei{\partial_t u_{k}, \phi_k} + \int_{Q_1} a_{ij}(t,x) D_j u_k D_i \phi_{k}\, dx\, dt = - \int_{Q_1} f_{kj}D_j \phi_k\, dtdx,
\]
for any $k =1, 2,\ldots, d$ and $\phi = (\phi_1, \phi_2,\ldots, \phi_d) \in C_0^\infty(Q_1)^d$ such that $\Div[\phi(t, \cdot)] =0$ for $t \in (-1, 0)$. On the other hand, a vector field $u \in W_{1,1}^{1,2}(Q_1)^d$ is said to be a strong solution of \eqref{eq7.41c} on $Q_1$ if \eqref{eq7.41c} holds for a.e. $(t,x) \in Q_1$ for some $p \in L_1(Q_1)$ with $\nabla p \in L_1(Q_1)^d$.

We are ready to state the main results of the paper. Our first theorem is about the $L_{s,q}$-estimate for gradients of weak solutions to \eqref{eq7.41b}.
Ò\begin{theorem}
                                \label{thm2.3}
Let $s,q\in (1,\infty)$, $\nu \in (0,1)$, and $\alpha_0\in (\min(s,q)/(\min(s,q)-1),\infty)$.  There exists $\delta=\delta(d,\nu,s,q,\alpha_0)$ such that the following statement holds. Assume that \eqref{ellipticity}-\eqref{symmetry} and Assumption \ref{assump1} $(\delta, \alpha_0)$ hold.  Assume also that $d_{ij} \in L_{s', q'}(Q_1)$ with $s', q'$ being as in \eqref{conjugate} and $i, j = 1,2,\ldots, d$. Then, if $(u,p)\in \cH^1_{s,q}(Q_1)^d\times L_{1}(Q_1)$ is a weak solution to \eqref{eq7.41b} in $Q_1$, $f  \in L_{s,q}(Q_1)^{d\times d}$, and $g \in L_{s,q}(Q_1)$, it holds that
\begin{equation} \label{main-est-1}
\begin{split}
\|Du\|_{L_{s, q}(Q_{1/2})} & \le N(d,\nu,s,q, \alpha_0)\Big[ \|f\|_{L_{s,q}(Q_{1})} +  \|g\|_{L_{s,q}(Q_{1})}\Big] \\
& \quad +N(d,\nu,s,q,R_0, \alpha_0)\|u\|_{L_{s,q}(Q_{1})}.
\end{split}
\end{equation}
\end{theorem}

Similar to Theorem \ref{thm2.3}, we also obtain the following mixed-norm regularity estimates for solutions of the non-divergence form Stokes system \eqref{eq7.41c}.
\begin{theorem}
                                \label{thm2.3b}
Let $s, q\in (1,\infty)$ and $\nu \in (0,1)$. There exists $\delta=\delta(d,\nu,s,q) \in (0,1)$ such that the following statement holds. Suppose that $d_{ij} =0$, the ellipticity condition \eqref{ellipticity} and Assumption \ref{assump1} $(\delta, 1)$ hold.  Then, if $u\in W^{1,2}_{s,q}(Q_1)^d$ is a strong solution to \eqref{eq7.41b} in $Q_1$, $f \in L_{s,q}(Q_1)^{d\times d}$, and $D g \in L_{s,q}(Q_1)^d$, then  it follows that
\begin{equation} \label{main-est-2}
\begin{split}
\|D^2u\|_{L_{s, q}(Q_{1/2})} & \le N(d,\nu,q)\Big[ \|f\|_{L_{s,q}(Q_{1})} +  \|Dg\|_{L_{s,q}(Q_{1})}\Big]\\
& \quad +N(d,\nu,q,R_0)\|u\|_{L_{s,q}(Q_{1})}.
\end{split}
\end{equation}
\end{theorem}
\begin{remark} By using interpolation and a standard iteration argument, \eqref{main-est-1} and \eqref{main-est-2} still hold if we replace the term $\norm{u}_{L_{s,q}(Q_1)}$ on the right-hand sides  by $\norm{u}_{L_{s,1}(Q_1)}$.
\end{remark}


Several remarks regarding our Theorems \ref{thm2.3} and \ref{thm2.3b} are in order. First of all, even when $q=s =2$ and $g \equiv 0$, the estimates \eqref{main-est-1} and \eqref{main-est-2} are  already new for the non-stationary Stokes system with variable coefficients. These estimates are known as Caccioppoli's type estimates. When $a_{ij}=\delta_{ij}$, $f\equiv 0$, and $g\equiv 0$, Caccioppoli type estimates for Stokes system were established in  \cite{Bum} by using special test functions. However, it is not so clear that this method can be extended to systems with variable coefficients and nonzero right-hand side.

Next, we emphasize that the estimates in Theorems \ref{thm2.3} and \ref{thm2.3b}  do not contain any pressure term on the right-hand sides. Thus, our results seem to be new even when the coefficients are constants. One can easily see that the estimates in Theorems \ref{thm2.3} and \ref{thm2.3b} imply the available regularity estimates such as \cite[Proposition 6.7, p. 84]{Seregin} in which the regularity for the pressure $p$ is required.

We note that $L_q$-estimates for non-stationary Stokes system with constant coefficients were established in \cite{Solonnikov} many years ago, and recently in \cite{DLW} with different approach. For stationary Stokes system with variable, \textup{VMO} or partially \textup{VMO} coefficients, both interior and boundary estimates were studied recently in \cite{CL17, Dong-Kim, Dong-Kim-Bull}, where slightly more general operators but with bounded coefficients were considered. However, the approaches used in these papers do not seem to be applicable to the non-stationary Stokes system.

Finally, we mention that the smallness Assumption \ref{assump1} $(\delta, \alpha_0)$ is necessary for both  Theorem \ref{thm2.3} and Theorem \ref{thm2.3b}. See an example in the well-known paper \cite{M} for linear elliptic equations in which $d_{ij} =0$, and an example in \cite{FP} in which $(a_{ij})$ is an identity matrix and $(d_{ij})$ is bounded but not small in the \textup{BMO} semi-norm.

Next, we give an application of our $L_{s,q}$-estimates for the Stokes system. Consider the Navier-Stokes equations
\begin{equation} \label{NS.eqn}
u_t - \Delta u + (u \cdot \nabla) u + \nabla p =0, \quad \Div u =0.
\end{equation}
Let $u$ be a Leray-Hopf weak solution of \eqref{NS.eqn} in $Q_1$.  For each $i, j = 1,2,\ldots, d$, let $ d_{ij}$ be the solution of  the  equation
\begin{equation} \label{d-matrix.def}
\left\{
\begin{array}{cccl}
\Delta d_{ij} & = & D_j u_i-D_i u_j & \quad \text{in}\quad  B_1 \\
d_{ij} & = & 0 & \quad \text{on} \quad \partial B_1.
\end{array} \right.
\end{equation}
Observe that for a.e. $t\in (-1, 0)$, we have $u(t, \cdot) \in L^2(B_1)$. Therefore, the existence and uniqueness of $d_{ij}(t,\cdot) \in W^{1,2}_0(B_1)$ follows, and the solution $d_{ij}(t,\cdot)$ satisfies the standard energy estimate, see \eqref{L-p-d.est}.  Let $[d_{ij}]_{B_{\rho}(x_0)}(t)$ be the average of $d_{ij}$ with respect to $x$ on $B_{\rho}(x_0)$.
As a corollary of Theorem \ref{thm2.3}, we obtain the following new $\epsilon$-regularity criterion for the Navier-Stokes equation \eqref{NS.eqn}.

\begin{theorem}
        \label{NS-reg.thm}
Let $\alpha_0\in (2(d+2)/(d+4),\infty)$. There exists $\epsilon \in (0,1)$ sufficiently small depending only on the dimension $d$ and $\alpha_0$ such that, if $u$ is a Leray-Hopf weak solution of \eqref{NS.eqn} in $Q_1$ and
\begin{equation} \label{epsilon-criteria}
\sup_{z_0 \in Q_{2/3}}\sup_{\rho \in (0, R_0)}\left( \fint_{Q_{\rho}(z_0)} |d_{ij}(t,x) - [d_{ij}]_{B_\rho(x_0)}(t)|^{\alpha_0}\,dx\,dt\right)^{1/{\alpha_0}} \leq \epsilon,
\end{equation}
for every  $i, j = 1,2,\ldots, d$ and for some $R_0 \in (0,1/2)$ and with $d_{ij}$ defined in \eqref{d-matrix.def}, then $u$ is smooth in $Q_{1/2}$.
\end{theorem}

The parameter $\alpha_0$ in the above theorem can be less than $2$, which might be useful in applications. We would like to note that many other $\epsilon$-regularity criteria for solutions to the Navier-Stokes equations were established, for instance, in \cite{CKN, GKT}. See also \cite[Chapter 6]{Seregin} for further discussion on this. To the best of our knowledge, compared to these known criteria, our result in Theorem \ref{NS-reg.thm} is completely new.
As an immediate consequence of Theorem \ref{NS-reg.thm}, we obtain the following regularity criteria for weak solutions to the Navier-Stokes equations, which implies Serrin's regularity criterion in the borderline case established by Fabes-Jones-Rivi\`{e}re \cite{FJR} and by Struwe \cite{Str88}.

\begin{corollary}
            \label{coro}
Assume that $u$ is a Leray-Hopf weak solution of \eqref{NS.eqn} in $Q_1$.

(i) Let $s,q\in (1,\infty]$ be such that $2/s+d/q = 1$. Suppose that $u \in L_s((-1,0);L^w_{q}(B_{1}))$ when $s<\infty$, or the $L_{\infty}((-1,0); L^w_{d}(B_1))$ norm of $u$ is sufficiently small.  Then, $u$ is smooth in $Q_{1/2}$.

(ii) Let $s,q\in (1,\infty]$ be such that $2/s+d/q = 1$. Suppose that $u \in L_s^w((-1,0);L^w_{q}(B_{1}))$ with a sufficiently small norm.  Then, $u$ is smooth in $Q_{1/2}$.

(iii) Let $\alpha\in [0,1)$, $\beta\in [0,d)$, and $s,q\in (1,\infty)$ be constants satisfying
\begin{equation}
                            \label{eq5.13}
\frac{2\alpha} s+\frac \beta q
=\frac 2 s+\frac d q-1(>0),\quad  \frac 1 s<\frac 1 2+\frac 1 {d+2},\quad \text{and}\quad \frac 1 q<\frac 1 2+\frac 1 {d+2}+\frac 1 d.
\end{equation}
Suppose that $u\in \mathscr{M}_{s,\alpha}((-1,0); \mathscr{M}_{q,\beta}(B_1))$ with a sufficiently small norm. Then, $u$ is smooth in $Q_{1/2}$.
\end{corollary}
Here $L_{q}^w$ denotes the weak-$L_s$ space, and $\mathscr{M}_{q,\beta}$ denotes the Morrey space
$$
\|f\|_{\mathscr{M}_{q,\beta}(B_1)}:=\left(\sup_{x_0\in \overline{B_1},\,r>0}r^{-\beta}\int_{B_r(x_0)\cap B_1}|f|^q\,dx\right)^{1/q}.
$$
Notice that in particular, when $d=3$, Corollary \ref{coro} (i) recovers a result by Kozono \cite{Ko98}. When $d=3$ and $q<\infty$, Corollary \ref{coro} (ii) was obtained in \cite{KK04}. Our approach only uses linear estimates and is very different from these in \cite{Ko98,KK04}.
It is also worth mentioning that we can take $q > 1$ and $s>10/7$ in Corollary \ref{coro} (iii) in the case when $d=3$.

We now briefly describe our methods in the proofs of the main results. Our approaches to prove Theorems \ref{thm2.3} and \ref{thm2.3b} are based on perturbation using equations with coefficients frozen in the spatial variable and sharp function technique introduced in \cite{Kry07, Krylov} and developed in \cite{DK16}. As we already mentioned, unlike in the stationary case studying in \cite{CL17, Dong-Kim, Dong-Kim-Bull}, even when $s=q =2$, the estimates \eqref{main-est-1} and \eqref{main-est-2} are not available to start the perturbation process. Our main idea to overcome this is to use the equations of vorticity,  which is in the spirit of Serrin \cite{Serrin}. Therefore, we need to derive several necessary estimates for the vorticity, and then, use the divergence equation and these estimates to derive desired estimates for the solutions. To prove Theorem \ref{NS-reg.thm}, we first rewrite the Navier-Stokes equations \eqref{NS.eqn} into a Stokes system in divergence form \eqref{eq7.41b} with coefficients that have singular skew-symmetric part $(d_{ij})$ defined in \eqref{d-matrix.def}. Then, we iteratively apply Theorem \ref{thm2.3} and the Sobolev embedding theorem to successively improve the regularity of weak solutions.

The rest of the paper is organized as follow. In Section \ref{Preli}, we recall several estimates for sharp functions, and derive necessary estimates of solution and its vorticity for Stokes systems with coefficients that only depend on the time variable. Section \ref{div-part} is devoted to prove Theorem \ref{thm2.3}, while  the proof of Theorem \ref{thm2.3b} is presented in Section \ref{non-div-se}. Finally, in the last section, Section \ref{NS-se}, we provide the proof of Theorem \ref{NS-reg.thm} as well as the proof of Corollary \ref{coro}.

\section{Preliminary estimates} \label{Preli}
\subsection{Sharp function estimates}
The following result is a special case of \cite[Theorem 2.3 (i)]{DK16}. Let $\cX\subset \bR^{d+1}$ be a space of homogeneous type, which is endowed with the parabolic distance and a doubling measure $\mu$ that is naturally inherited from the Lebesgue measure. As in \cite{DK16}, we take a filtration of partitions of $\cX$ (cf. \cite{MR1096400}) and for any $f\in L_{1,\text{loc}}$ we define its dyadic sharp function $f^{\#}_{\text{dy}}$ in $\cX$ associated with the filtration of partitions. Also for each $q \in [1, \infty]$, $A_q$ denotes the Muckenhoupt class of weights.

\begin{theorem} \label{Feffer}
Let $s, q \in (1,\infty)$, $K_0 \geq 1$ and $\omega \in A_{q}$ with $[\omega]_{A_q} \leq K_0$.  Suppose that $f \in L_{s}(\omega d\mu)$.
Then,
\[
\norm{f}_{L_s(\omega d\mu)} \leq  N\left[\norm{f^{\#}_{\text{dy}}}_{L_{s}(\omega d\mu)} + \mu(\mathcal{X})^{-1} \omega(\textup{supp}(f))^{\frac{1}{s}} \norm{f}_{L_1(\mu)} \right],
\]
where $N>0$ is a constant depending only on $s$, $q$, $K_0$, and the doubling constant of $\mu$.
\end{theorem}
The following lemma is a direct corollary of Theorem \ref{Feffer}.
\begin{lemma}
    \label{mixed-norm-lemma}
For any $s,q \in (1, \infty)$, there exists a constant $N = N(d, s, q) >0$ such that
\[
\norm{f}_{L_{s,q}(Q_R)} \leq N\Big[ \norm{I_{Q_R}f^{\#}_{\textup{dy}}}_{L_{s,q}(Q_R)} +   R^{\frac{2}{s}+\frac {d}{q} - d-2} \norm{f}_{L_1(Q_R)}\Big].
\]
for any $R>0$ and $f \in L_{s,q}(Q_R)$.
\end{lemma}
\begin{proof} For $t\in (-R^2,0)$, let
\[
\psi(t) = \norm{f(t, \cdot)}_{L_{q}(B_R)} \quad \text{and} \quad
\phi(t) = \norm{f^{\#}_{\textup{dy}}(t,\cdot)+(|f|)_{Q_R}}_{L_{q}(B_R)}.
\]
Moreover, for any $\omega \in A_{q}((-R^2,0))$ with $[\omega]_{A_{q}} \leq K_0$, we write $\tilde{\omega}(t,x) = \omega(t)$ for all $(t,x) \in Q_R$.
Then, by applying Theorem \ref{Feffer} with $\cX=Q_R$, we obtain
\[
\norm{\psi}_{L_{q}((-R^2,0), \omega)}= \norm{f}_{L_{q}(Q_R, \tilde{\omega})}\le N\norm{f^\#_{\text{dy}}+(|f|)_{Q_R}}_{L_{q}(Q_R, \tilde{\omega})}
=  N \norm{\phi}_{L_{q}((-R^2,0), \omega)},
\]
with $N = N(d, K_0, s)$. 
Then, by the extrapolation theorem (see, for instance, \cite[Theorem 2.5]{DK16}), we see that
\[
\norm{\psi}_{L_{s}((-R^2,0), \omega)} \leq 4N \norm{\phi}_{L_{s}((-R^2,0), \omega)}, \quad \forall \ \omega \in A_{s}, \quad [\omega]_{A_{s}} \leq K_0.
\]
Note that in the special case when $\omega\equiv 1$, $\norm{\psi}_{L_{s}((-R^2,0), \omega)}  = \norm{f}_{L_{s,q}(Q_R)}$ and
\[
 \norm{\phi}_{L_{s}((-R^2,0), \omega)} \le \norm{f^{\#}_{\text{dy}}}_{L_{s,q}(Q_R)} + R^{2/s+d/q} (|f|)_{Q_R}.
\]
Therefore, the desired estimate follows.
\end{proof}
\subsection{Stokes systems with simple coefficients}
In this subsection, we consider the time-dependent Stokes system with coefficients that only depend on the time variable
\begin{equation}
                        \label{eq7.41}
u_t-D_{i}(a_{ij}(t)D_j u)+\nabla p=0,\quad \Div u= 0,
\end{equation}
where $a_{ij}=  b_{ij}(t) + d_{ij}(t)$ with $b_{ij} = b_{ji}$ and $d_{ij} = -d_{ji}$ for all $i, j = \{1, 2,\ldots, d\}$. Moreover, $a_{ij}$ satisfies the ellipticity condition with ellipticity constant $\nu\in (0,1)$: for any $\xi\in \bR^d$,
\begin{equation}
                        \label{eq8.13}
\nu |\xi|^2\le b_{ij}\xi_i\xi_j,\quad |b_{ij}|\le \nu^{-1}.
\end{equation}
We have the following gradient estimate.
\begin{lemma}
                \label{lem1.2}
Assume that \eqref{eq8.13} holds. Let $q_0\in (1,\infty)$, and $(u,p)\in\cH^1_{q_0}(Q_1)^d\times L_{1}(Q_1)$ be a weak solution to \eqref{eq7.41} in $Q_1$. Then we have
\begin{equation}
                                        \label{eq7.46}
\|D^2u\|_{L_{q_0}(Q_{1/2})} +  \|Du\|_{L_{q_0}(Q_{1/2})}\le N(d,\nu,q_0)\|u-[u]_{B_1}(t)\|_{L_{q_0}(Q_{1})},
\end{equation}
where $[u]_{B_1}(t)$ is the average of $u(t,\cdot)$ in $B_1$.
\end{lemma}
\begin{proof}
By a mollification in $x$, we see that $\omega=\nabla \times u$ is a weak solution to the parabolic equation
$$
\omega_t-D_{i}(a_{ij}(t)D_j \omega)=0  \quad \text{in} \quad Q_1.
$$
Observe that since the matrix $(d_{ij}(t))_{n \times n}$ is skew-symmetric, $\omega$ is indeed a weak solution of
\[
\omega_t-D_{i}(b_{ij}(t)D_j \omega)=0 \quad \text{in} \quad Q_1.
\]
Since the matrix $(b_{ij})_{n\times n}$ satisfies the ellipticity condition as in \eqref{eq8.13}, we can apply the local $\cH^1_p$ estimate for linear parabolic equations to obtain
\begin{equation}
                            \label{eq7.57}
\|D\omega\|_{L_{q_0}(Q_{2/3})}\le N(d,\nu,q_0)\|\omega\|_{L_{q_0}(Q_{3/4})}.
\end{equation}
Since $u$ is divergence free, we have
$$
\Delta u_i=  -D_i  \sum_{k=1}^d D_k u_k + \sum_{k= 1 }^d D_{kk}u_i
=  \sum_{k\neq i}D_k(D_k u_i-D_i u_k).
$$
Thus by the local $W^1_p$ estimate for the Laplace operator,
$$
\|Du\|_{L_{q_0}(Q_{1/2})}\le N \|\omega\|_{L_{q_0}(Q_{2/3})}+ N\|u\|_{L_{q_0}(Q_{2/3})}.
$$
Similarly,
\begin{align*}
&\|D^2u\|_{L_{q_0}(Q_{1/2})}\le N\|D\omega\|_{L_{q_0}(Q_{2/3})}+N\|Du\|_{L_{q_0}(Q_{2/3})}
\le N\|Du\|_{L_{q_0}(Q_{3/4})}\\
&\le \varepsilon\|D^2u\|_{L_{q_0}(Q_{3/4})}+N\varepsilon^{-1}\|u-[u]_{B_1}(t)\|_{L_{q_0}(Q_{3/4})}
\end{align*}
for any $\varepsilon\in (0,1)$, where we used \eqref{eq7.57} in the second inequality, and multiplicative inequalities in the last inequality.
It then follows from a standard iteration argument that
$$
\|D^2u\|_{L_{q_0}(Q_{1/2})}
\le N\|u-[u]_{B_1}(t)\|_{L_{q_0}(Q_{1})},
$$
from which and multiplicative inequalities we obtain \eqref{eq7.46}.
The lemma is proved.
\end{proof}
Recall that for each $\alpha \in (0, 1]$, and each parabolic cylinder $Q \in \bR^{n+1}$, we write
\[
[[u]]_{C^{\alpha/2, \alpha}(Q)} = \sup_{\substack{(t,x), (s,y) \in Q\\ (t,x) \not=(s,y)}} \frac{|u(t,x)-u(s,y)|}{|t-s|^{\alpha/2} + |x-y|^{\alpha}},
\]
and
\[
\|u\|_{C^{\alpha/2, \alpha}(Q)} = \|u\|_{L_\infty(Q)} + [[u]]_{C^{\alpha/2, \alpha}(Q)}.
\]
\begin{lemma}
                    \label{lem1.3}
Under the assumptions of Lemma \ref{lem1.2}, we have
$$
\|\omega\|_{C^{1/2,1}(Q_{1/2})}\le N(d,\nu,q_0)\|\omega\|_{L_{q_0}(Q_{1})},
$$
where $\omega=\nabla\times u$.
\end{lemma}
\begin{proof}
The lemma follows by using mollifications in $x$ and the standard interior estimate for parabolic equations.
\end{proof}


\section{Divergence form Stokes system and proof of Theorem \ref{thm2.3}} \label{div-part}

Note that for each integrable function $f$ defined in a measurable set $Q \subset \bR^{d+1}$,  $(f)_Q$ denotes the average of $f$ in $Q$, i.e.,
\[
(f)_{Q} = \fint_{Q} f(t,x) \,dx\,dt.
\]
We need to establish several lemmas in order to prove Theorem \ref{thm2.3}. Our first lemma gives the control of $(|Du|^{q_0})_{Q_{r/2}}^{1/q_0}$ for weak solution $u$ of the Stokes system \eqref{eq7.41b}.
\begin{lemma}
                    \label{lem2.4b}
Let $\delta, \nu \in (0,1)$, $q_0\in (1,\infty)$, $q\in (q_0,\infty)$. Suppose that \eqref{ellipticity}-\eqref{symmetry} hold, and Assumption \ref{assump1} $(\delta, \alpha_0)$ holds with $\alpha_0 \geq \frac{q_0q}{q-q_0}$. Then, for any $r\in (0,R_0)$ and weak solution $(u,p)\in \cH^1_{s,q}(Q_{r})^d\times L_{1}(Q_{r})$  of \eqref{eq7.41b} in $Q_r$, we have
\begin{align*}
(|Du|^{q_0})_{Q_{r/2}}^{1/q_0}
&\le N(d,\nu,q_0)\Big((|f|^{q_0})_{Q_{r}}^{1/q_0}
+r^{-1}(|u-[u]_{B_r}(t)|^{q_0})_{Q_{r}}^{1/q_0}\Big)\\
&\quad +N(d,\nu,q_0) \delta (|Du|^{q})_{Q_{r}}^{1/q} + N(d, q_0) (|g|^{q_0})_{Q_{r}}^{1/{q_0}}.
\end{align*}
\end{lemma}
\begin{proof} Let $(w,p_1)$ be a weak solution to
$$
w_t-D_{i}(\bar a_{ij}(t)D_j w)+\nabla p_1=\Div (I_{Q_r}f)+D_i(I_{Q_r}(a_{ij}-\bar a_{ij})D_j u),\quad \Div w= I_{Q_r} g
$$
in $(-r^2,0)\times \bR^d$ with the zero initial condition on $\{t=-r^2\}$. Then $\nabla\times w$ is a so called {\em adjoint solution} to the parabolic equation. By duality and the equation $\Div w= I_{Q_r} g$,
\begin{align*}
&\|Dw\|_{L_{q_0}((-r^2,0)\times \bR^d)}
\le N(d,q_0)\Big [\|\nabla\times w\|_{L_{q_0}((-r^2,0)\times \bR^d)} + \|I_{Q_r} g \|_{L_{q_0}((-r^2,0)\times \bR^d)} \Big]\\
&\le N(d,\nu,q_0)\|f\|_{L_{q_0}(Q_r)}
+N(d,\nu,q_0)\|(a_{ij}-\bar a_{ij})D_j u\|_{L_{q_0}(Q_r)} +N(d, q_0) \|g\|_{L_{q_0}(Q_r)}.
\end{align*}
Thus, we have
\begin{align}
                            \label{eq8.37b}
&(|Dw|^{q_0})_{Q_{r}}^{1/{q_0}}\le N(d,\nu,q_0)\Big[ (|f|^{q_0})_{Q_{r}}^{1/{q_0}} + (|(a_{ij}-\bar a_{ij})D_j u|^{q_0})_{Q_{r}}^{1/{q_0}}\Big] +  N(d,q_0)(|g|^{q_0})_{Q_{r}}^{1/{q_0}}\nonumber\\
&\le N(d,\nu,q_0)\Big[ (|f|^{q_0})_{Q_{r}}^{1/{q_0}}+
\delta(|Du|^{q})_{Q_{r}}^{1/q}\Big] +  N(d, q_0) (|g|^{q_0})_{Q_{r}}^{1/{q_0}},
\end{align}
where we used Assumption \ref{assump1} with $\alpha_0 \geq \frac{q_0 q}{q-q_0}$ and H\"older's inequality for the middle term on the right-hand side in the last inequality. Now $(v,p_2):=(u-w,p-p_1)$ satisfies
$$
v_t-D_{i}(\bar a_{ij}(t)D_j v)+\nabla p_2=0,\quad \Div v=0
$$
in $Q_r$. By Lemma \ref{lem1.2} with a scaling, we have
\begin{equation}
                            \label{eq8.43b}
(|Dv|^{q_0})_{Q_{r/2}}^{1/{q_0}}\le r^{-1}(|v-[v]_{B_{r}}(t)|^{q_0})_{Q_{r}}^{1/{q_0}}.
\end{equation}
By \eqref{eq8.37b}, \eqref{eq8.43b}, the triangle inequality, and the Poincar\'e inequality, we get the desired inequality.
\end{proof}

In the next lemma we prove a mean oscillation estimate of $\nabla \times u$.

\begin{lemma} \label{div-vorticity-oss.est}
Let $q_1 \in (1, \infty), q_0 \in (1, q_1)$, $\delta \in (0,1)$, $R_0\in (0,1/4)$, $r \in (0, R_0)$, $\kappa \in (0, 1/2)$. Assume that \eqref{ellipticity}-\eqref{symmetry} hold, and Assumption \ref{assump1} ($\delta, \alpha_0$) holds with $\alpha_0 \ge \frac{q_0 q_1}{q_1-q_0}$. Suppose that $(u,p)\in \cH^1_{s_1,q_1}(Q_r)^d \times L_{1}(Q_r)$ is a weak solution to \eqref{eq7.41b} in $Q_r$. Then it holds that
\begin{align*}
(\omega -(\omega)_{Q_{\kappa r}})_{Q_{\kappa r}} & \leq N(d, \nu, q_0) \kappa^{-\frac{n+2}{q_0}}(|f|^{q_0})_{Q_{r}}^{1/q_0} + N(d, q_0) \kappa (|g|^{q_0})_{Q_{r}}^{1/{q_0}} \\
& \quad \quad + N(n, \nu, q_0, q_1)\big(\kappa^{-\frac{n+2}{q_0}} \delta + \kappa\big) (|Du|^{q_1})_{Q_{r}}^{1/q_1},
\end{align*}
where $\omega = \nabla \times u$.
\end{lemma}
\begin{proof}
Let $(w, p_1)$ and $(v, p_2)$ be as in the proof of Lemma \ref{lem2.4b}. In particular,  $(w, p_1)$ is a weak solution of
\[
w_t - D_i(\bar{a}_{ij}(t) D_j w) + \nabla p_1 = \Div[I_{Q_{r}} f] + D_i(I_{Q_{r}} (a_{ij} - \bar{a}_{ij}(t)) D_j u), \quad \Div w = I_{Q_r} g
\]
in $(-r^2, 0) \times \bR^{d}$ with zero initial condition on $\{t = -r^2\}$. Also, $(v, p_2) = (u-w, p- p_1)$ is a weak solution of
\[
v_t - D_i(\bar{a}_{ij}(t) D_j v) + \nabla p_2 = 0, \quad \Div v =0
\]
in $Q_{r}$. Let $\omega_1 = \nabla \times w$ and $\omega_2 = \nabla \times v$. Observe that $\omega = \omega_1 + \omega_2$. Moreover, from the definition of $\omega_1$ and \eqref{eq8.37b},
\begin{align}
(|\omega_1|^{q_0})_{Q_{r}}^{1/q_0} & \leq (|D w|^{q_0})_{Q_{r}}^{1/q_0} \nonumber \\
 \label{omega-1.est}
& \leq N(d, \nu, q_0) \Big[ (|f|^{q_0})_{Q_{r}}^{1/q_0} + \delta (|Du|^{q_1})_{Q_{r}}^{1/q_1}\Big] + N(d, q_0) (|g|^{q_0})_{Q_{r}}^{1/{q_0}}.
\end{align}
On the other hand, by applying Lemma \ref{lem1.3} to $\omega_2$ with suitable scaling, we obtain
\begin{align*}
(\omega_2 - (\omega_2)_{Q_{\kappa r}})_{Q_{\kappa r}} & \leq N\kappa r [[\omega_2]]_{C^{1/2, 1}(Q_{r/2})} \leq N(d, \nu, q_0) \kappa (|\omega_2|^{q_0})_{Q_{r}}^{1/q_0} \nonumber \\
& \leq N(d, \nu, q_0) \kappa \Big[ (|\omega|^{q_0})_{Q_{r}}^{1/q_0} + (|\omega_1|^{q_0})_{Q_{r}}^{1/q_0}\Big].
\end{align*}
We then combine the last estimate with \eqref{omega-1.est} and the fact that $\delta \in (0,1)$ to deduce that
\begin{equation} \label{omega-2.est}
\begin{split}
& (\omega_2 - (\omega_2)_{Q_{\kappa r}})_{Q_{\kappa r}} \\
& \leq \kappa  \Big[ N(d, \nu, q_0) (|f|^{q_0})_{Q_{r}}^{1/q_0} + N(d, \nu, q_0) (|Du|^{q_1})_{Q_{r}}^{1/q_1} +  N(d, q_0) (|g|^{q_0})_{Q_{r}}^{1/{q_0}} \Big] .
\end{split}
\end{equation}
Moreover, by using the inequality
\[
\fint_{Q_{\kappa r}}|\omega -(\omega)_{Q_{\kappa r}}| \,dx\, dt \leq 2 \fint_{Q_{\kappa r}}|\omega -c| \,dx\,dt
\]
with $c = (\omega_2)_{Q_{\kappa r}}$, and then applying
the triangle inequality and H\"{o}lder's inequality, we have
\begin{align*}
&\fint_{Q_{\kappa r}}|\omega -(\omega)_{Q_{\kappa r}}| \,dx\, dt   \leq 2 \fint_{Q_{\kappa r}} |\omega - (\omega_2)_{Q_{\kappa r}}| \,dx \, dt\\
& \leq 2  \fint_{Q_{\kappa r}} |\omega_2 - (\omega_2)_{Q_{\kappa r}}| \,dx\, dt  + N(d, q_0) \kappa^{-\frac{d+2}{q_0}}\left(\fint_{Q_{r}} |\omega_1|^{q_0} \,dx\, dt\right)^{1/q_0}.
\end{align*}
This last estimate together with  \eqref{omega-1.est} and \eqref{omega-2.est} gives that
\begin{align*}
(\omega -(\omega)_{Q_{\kappa r}})_{Q_{\kappa r}} & \leq N(d, \nu, q_0)\Big( \kappa^{-\frac{d+2}{q_0}}+\kappa\Big)(|f|^{q_0})_{Q_{r}}^{1/q_0}  + N(d, q_0) \kappa (|g|^{q_0})_{Q_{r}}^{1/{q_0}} \\
& \quad \quad + N(d, \nu, q_0)\Big(\kappa^{-\frac{d+2}{q_0}} \delta + \kappa\Big) (|Du|^{q_1})_{Q_{r}}^{1/q_1},
\end{align*}
which  implies our desired estimate as $\kappa \in (0, 1/2)$.
\end{proof}

Our next lemma gives the key estimates of vorticity $\omega = \nabla \times u$ and $Du$ in the mixed norm.
\begin{lemma}
                            \label{lem4.10}
Let $R\in (0,R_0)$, $\delta \in (0,1)$, $\kappa \in (0,1/2)$, $s, q \in (1, \infty)$,  and
$$
\alpha_0\in (\min(s,q)/(\min(s,q)-1),\infty).
$$
Assume that \eqref{ellipticity}-\eqref{symmetry} hold and Assumption \ref{assump1} ($\delta, \alpha_0$) is satisfied.
Suppose that $(u,p)\in \cH^1_{s,q}(Q_R)^d\times L_{1}(Q_R)$ is a weak solution to \eqref{eq7.41b} in $Q_R$, and $\omega = \nabla \times u$. Then
we have
\begin{align}
&\norm{\omega}_{L_{s,q}(Q_{2R/3})} \leq N \kappa^{-\frac{d+2}{q_0}}\norm{f}_{L_{s,q}(Q_{3R/4})} + N\kappa \norm{g}_{L_{s,q}(Q_{3R/4})} \nonumber\\
                            \label{eq1.39}
&\quad+N\Big(\kappa^{-\frac{d+2}{q_0}} \delta + \kappa\Big) \norm{Du}_{L_{s,q}(Q_{3R/4})}+ N  R^{2/s+d/q}\kappa^{-d-2}(|\omega|)_{Q_{3R/4}},
\end{align}
and
\begin{align}
&\norm{Du}_{L_{s,q}(Q_{R/2})} \leq N \kappa^{-d-2}\norm{f}_{L_{s,q}(Q_{R})} + N \kappa^{-d-2}\norm{g}_{L_{s,q}(Q_{R})} \nonumber\\
                            \label{eq1.40}
&\quad+N\Big(\kappa^{-d-2} \delta + \kappa\Big) \norm{Du}_{L_{s,q}(Q_{R})}+
N \kappa^{-d-2}R^{-1}\norm{u}_{L_{s,q}(Q_{R})}.
\end{align}
\end{lemma}
\begin{proof}
Take $q_1 \in (1, \min(s,q))$ and $q_0\in (1,q_1)$ such that
$\alpha_0\ge q_0q_1/(q_1-q_0)$.
We consider two cases.
\\ \noindent
{\em Case 1: $r \in (0, R/12)$.} It follows from Lemma \ref{div-vorticity-oss.est} that for all $z_0 \in Q_{2R/3}$,
\begin{align*}
(\omega -(\omega)_{Q_{\kappa r}(z_0)})_{Q_{\kappa r}(z_0)} &\leq N(d, \nu, q_0) \kappa^{-\frac{d+2}{q_0}} (|f|^{q_0})^{1/q_0}_{Q_{r}(z_0)} +  N(d, q_0)\kappa (|g|^{q_0})_{Q_{r}}^{1/{q_0}} \\
& \quad + N(d, \nu, q_0)\Big(\kappa^{-\frac{d+2}{q_0}} \delta+ \kappa\Big)(|Du|^{q_1})^{1/q_1}_{Q_{r}(z_0)}.
\end{align*}
Observe that because $r < R/12$, we have $Q_{r}(z_0) \subset Q_{3R/4}$.
Therefore,
\[
\begin{split}
& (|f|^{q_0})^{1/q_0}_{Q_{r}(z_0)} \leq \mathcal{M}(I_{Q_{_{3R/4}}}|f|^{q_0})^{1/q_0}(z_0), \\ 
&  (|g|^{q_0})_{Q_{r}}^{1/{q_0}} \leq  \mathcal{M}(I_{Q_{_{3R/4}}}|g|^{q_0})^{1/q_0}(z_0), \quad \text{and}  \\
& (|Du|^{q_1})^{1/q_1}_{Q_{r}(z_0)}  \le \mathcal{M}(I_{Q_{_{3R/4}}}|Du|^{q_1})^{1/q_1}(z_0),
\end{split}
\]
which imply that
\begin{align*}
(\omega -(\omega)_{Q_{\kappa r}(z_0)})_{Q_{\kappa r}(z_0)}
& \leq N \kappa^{-\frac{d+2}{q_0}} \mathcal{M}(I_{Q_{_{3R/4}}}|f|^{q_0})^{1/q_0}(z_0) + N \kappa \mathcal{M}(I_{Q_{_{3R/4}}}|g|^{q_0})^{1/q_0}(z_0) \\
&\quad + N\Big(\kappa^{-\frac{d+2}{q_0}} \delta+ \kappa\Big)\mathcal{M}(I_{Q_{_{3R/4}}}|Du|^{q_1})^{1/q_1}(z_0)
.
\end{align*}
\noindent
{\em Case 2: $r\in [R/12, R/(12\kappa))$.} In this case, we simply estimate
$$
(\omega -(\omega)_{Q_{\kappa r}(z_0)})_{Q_{\kappa r}(z_0)}
\le 2(|\omega|)_{Q_{\kappa r}(z_0)}
\le N\kappa^{-d-2}(|\omega|)_{Q_{3R/4}}.
$$
Now we take $\cX=Q_{2R/3}$ and define the dyadic sharp function $\omega^\#_{\text{dy}}$ of $\omega$ in $\cX$. From the above two cases, we conclude that for any $z_0\in \cX$,
\begin{align*}
\omega_{\text{dy}}^{\#}(z_0)   & \leq N(d, \nu, q_0) \kappa^{-\frac{d+2}{q_0}} \mathcal{M}(I_{Q_{_{3R/4}}}|f|^{q_0})^{1/q_0}(z_0) +  N(d, q_0) \kappa \mathcal{M}(I_{Q_{_{3R/4}}}|g|^{q_0})^{1/q_0}(z_0)\\
& \quad  + N(d, \nu, q_0)\Big(\kappa^{-\frac{d+2}{q_0}} \delta+ \kappa\Big) \mathcal{M}(I_{Q_{_{3R/4}}}|Du|^{q_1})^{1/q_1}(z_0) + N\kappa^{-d-2}(|\omega|)_{Q_{3R/4}}.
\end{align*}
Recalling that $1<q_0<q_1<\min(s,q)$, by Lemma \ref{mixed-norm-lemma} and the Hardy-Littlewood maximum function theorem in mixed-norm spaces (see, for instance, \cite[Corollary 2.6]{DK16}),
\begin{align*}
& \norm{\omega}_{L_{s,q}(Q_{2R/3})}  \leq N(d,s,q)\Big[ \norm{\omega_{\text{dy}}^{\#}}_{L_{s,q}
(Q_{2R/3})} + R^{\frac{2}{s} + \frac{d}{q} -d-2} \norm{\omega}_{L_1(Q_{2R/3})} \Big]\\
& \leq  N \kappa^{-\frac{d+2}{q_0}} \norm{\mathcal{M}(I_{Q_{_{3R/4}}}|f|^{q_0})^{1/q_0}}_{L_{s,q}(\bR^{d+1})} + N \kappa \norm{\mathcal{M}(I_{Q_{_{3R/4}}}|g|^{q_0})^{1/q_0}}_{L_{s,q}(\bR^{d+1})} \\
&\quad+N \Big(\kappa^{-\frac{d+2}{q_0}} \delta+ \kappa\Big) \norm{\mathcal{M}(I_{Q_{_{3R/4}}}|Du|^{q_1})^{1/q_1}}_{L_{s,q}(\bR^{d+1})} + NR^{2/s+d/q}\kappa^{-d-2}(|\omega|)_{Q_{3R/4}}\\
&\le N \Big[ \kappa^{-\frac{d+2}{q_0}}\norm{f}_{L_{s,q}(Q_{3R/4})} + \kappa \norm{g}_{L_{s,q}(Q_{3R/4})} \\
& \quad  +\Big(\kappa^{-\frac{d+2}{q_0}} \delta + \kappa\Big) \norm{Du}_{L_{s,q}(Q_{3R/4})}+  R^{2/s+d/q}\kappa^{-d-2}(|\omega|)_{Q_{3R/4}}\Big],
\end{align*}
where $N = N(d, \nu, s, q, q_0, q_1)$. This estimate gives \eqref{eq1.39}.

Next we show \eqref{eq1.40}. Since $\Div u =g$, as in the proof of Lemma \ref{lem1.2}, we have
\begin{equation}
                            \label{eq1.52}
\|Du\|_{L_{s,q}(Q_{R/2})}
\le N\|\omega\|_{L_{s,q}(Q_{2R/3})} +  N \norm{g}_{L_{s,q}(Q_{2R/3})} +NR^{-1}\|u\|_{L_{s,q}(Q_{2R/3})}.
\end{equation}
We also use H\"older's inequality and Lemma \ref{lem2.4b} with a covering argument to estimate the last term in \eqref{eq1.39} by
\begin{align} \nonumber
& R^{2/s+d/q}\kappa^{-d-2}(|\omega|)_{Q_{3R/4}} \\
& \le NR^{2/s+d/q}\kappa^{-d-2}(|Du|^{q_0})^{1/q_0}_{Q_{3R/4}}\nonumber \\
&\le NR^{2/s+d/q}\kappa^{-d-2}\Big((|f|^{q_0})_{Q_{R}}^{1/q_0} + (|g|^{q_0})_{Q_{R}}^{1/q_0}
+R^{-1}(|u|^{q_0})_{Q_{R}}^{1/q_0}+\delta (|Du|^{q_1})_{Q_{R}}^{1/q_1}\Big)\nonumber\\
                    \label{eq2.01}
&\le N\kappa^{-d-2}\Big(\|f\|_{L_{s,q}(Q_{R})} +  \|g\|_{L_{s,q}(Q_{R})}
+R^{-1}\|u\|_{L_{s,q}(Q_{R})}+\delta\|Du\|_{L_{s,q}(Q_{R})}\Big).
\end{align}
Combining \eqref{eq1.52}, \eqref{eq1.39}, and \eqref{eq2.01}, we reach \eqref{eq1.40}. The lemma is proved.
\end{proof}

Now we are ready to give the proof of Theorem \ref{thm2.3}.
\begin{proof}[Proof of Theorem \ref{thm2.3}]
For $k=1,2,\ldots$, we denote $Q^k=(-(1-2^{-k})^2,0)\times B_{1-2^{-k}}$. Let $k_0$ be the smallest positive integer such that $2^{-k_0-1}\le R_0$. For $k\ge k_0$, we apply \eqref{eq1.40} with $R=2^{-k-1}$ and a covering argument to get
\begin{align}
&\norm{Du}_{L_{s,q}(Q^{k})} \leq N \kappa^{-d-2}\norm{f}_{L_{s,q}(Q^{k+1})} + N  \kappa^{-d-2}\|g\|_{L_{s,q}(Q^{k+1})} \nonumber\\
                            \label{eq2.15}
&\quad+N\Big(\kappa^{-d-2} \delta + \kappa\Big) \norm{Du}_{L_{s,q}(Q^{k+1})}+
N \kappa^{-d-2}2^k\norm{u}_{L_{s,q}(Q^{k+1})}.
\end{align}
Note that the constants $N$ above are independent of $k$.
We then take $\kappa$ sufficiently small and then $\delta$ sufficiently small so that $N\Big(\kappa^{-d-2} \delta + \kappa\Big)\le 1/3$. Finally, we multiply both sides of \eqref{eq2.15} by $3^{-k}$ and sum in $k=k_0,k_0+1,\ldots$ to get the desired estimate. The theorem is proved.
\end{proof}

\section{Non-divergence form Stokes system and proof of Theorem \ref{thm2.3b}.} \label{non-div-se}

In this section, we consider the non-divergence form Stokes system and give the proof of Theorem \ref{thm2.3b}. The following lemma is analogous to Lemma \ref{lem2.4b}.

\begin{lemma}   \label{lem2.4c} Let $q_0\in (1,\infty)$, $q\in (q_0,\infty)$, $r\in (0,R_0)$,  $\nu, \delta\in (0,1)$, and $u\in W^{1,2}_q(Q_{r})^d$ be a strong solution to \eqref{eq7.41c} in $Q_r$. Suppose that \eqref{ellipticity} and Assumption \ref{assump1} ($\delta, 1$) hold. Then we have
\begin{align}
(|D^2u|^{q_0})_{Q_{r/2}}^{1/q_0}&\le N(d, q_0) (|Dg|^{q_0})_{Q_{r}}^{1/q_0}  + N(d,\nu,q_0,q) (|f|^{q_0})_{Q_{r}}^{1/q_0}
\nonumber\\
                                    \label{eq9.19}
&\quad +N(d,\nu,q_0,q)\Big[  r^{-1}(|Du-[Du]_{B_r}(t)|^{q_0})_{Q_{r}}^{1/q_0} + \delta^{1/q_0-1/q}(|D^2u|^q)_{Q_{r}}^{1/q}\Big].
\end{align}
\end{lemma}
\begin{proof}
The proof is similar to that of Lemma \ref{lem2.4b}.
Let $(w,p_1)$ be a strong solution to
$$
w_t-\bar a_{ij}(t)D_{ij} w+\nabla p_1=I_{Q_r}(f+(a_{ij}-\bar a_{ij})D_{ij} u)\,\quad \Div w=  \phi_r (g - [g]_{B_{r}}(t))
$$
in $(-r^2,0)\times \bR^d$ with zero initial condition on $\{t=-r^2\}$, where $\phi_r \in C_0^\infty((-r^2,r^2)\times B_{r})$ is a standard non-negative cut-off function, which satisfies $\phi_r = 1$ on $Q_{2r/3}$ and $|D\phi_r|\le 4/r$. Observe that from the equation $\Div w = \phi_{r} (g - [g]_{B_{r}}(t))$ and the Poincar\'e inequality, we have
\begin{equation} \label{non-est.w}
\begin{split}
& \|D^2w\|_{L_{q_0}((-r^2,0)\times \bR^d)} \\
& \leq N(d,q_0)\Big[ \|D\omega\|_{L_{q_0}((-r^2,0)\times \bR^d)} + \|D (\phi_{r} (g - [g]_{B_r}))\|_{L_{q_0}((-r^2,0)\times \bR^d)} \Big] \\
& \leq N(d,q_0)\Big[ \|D\omega\|_{L_{q_0}((-r^2,0)\times \bR^d)} + \|D  g\|_{L_{q_0}(Q_{r})} \Big].
\end{split}
\end{equation}
Now, $\omega:=\nabla\times w$ is a weak solution to the divergence form parabolic equation
$$
\omega_t-\bar a_{ij}(t)D_{ij} \omega =\nabla\times \big(I_{Q_r}(f+(a_{ij}-\bar a_{ij})D_{ij} u)\big).
$$
By applying the $\cH^1_p$ estimate for divergence form parabolic equations and \eqref{non-est.w}, we obtain
\begin{align*}
&\|D^2w\|_{L_{q_0}((-r^2,0)\times \bR^d)}\\
&\le N(d,\nu,q_0)\Big[ \|f\|_{L_{q_0}(Q_r)} + \|(a_{ij}-\bar a_{ij})D_{ij} u\|_{L_{q_0}(Q_r)} \Big]+  N(d, q_0)  \|Dg\|_{L_{q_0}(Q_{r})}.
\end{align*}
From this and by using  Assumption \ref{assump1} and H\"{o}lder's inequality for the middle term on the right hand side of the last estimate, we have
\begin{align}\nonumber
(|D^2w|^{q_0})_{Q_{r}}^{1/{q_0}}
& \leq N(d, q_0) (|Dg|^{q_0})_{Q_{r}}^{1/{q_0}}  \\  \label{eq8.37c}
&\quad + N(d,\nu,q_0)\Big[ (|f|^{q_0})_{Q_{r}}^{1/{q_0}}
+ \delta^{1/q_0-1/q}(|D^2u|^q)_{Q_{r}}^{1/q}\Big].
\end{align}
Now $(v,p_2):=(u-w,p-p_1)$ satisfies
$$
v_t-\bar a_{ij}(t)D_{ij} v+\nabla p_2=0,\quad \Div v= [g]_{B_{r}}(t)
$$
in $Q_{2r/3}$. By Lemma \ref{lem1.2} applied to $Dv$ with a scaling, we have
\begin{equation}
                            \label{eq8.43c}
(|D^2v|^{q_0})_{Q_{r/2}}^{1/{q_0}}\le r^{-1}(|Dv-[Dv]_{B_{r}}(t)|^{q_0})_{Q_{r}}^{1/{q_0}}.
\end{equation}
By \eqref{eq8.37c}, \eqref{eq8.43c}, the triangle inequality, and the Poincar\'e inequality, we get the desired inequality.
\end{proof}
\begin{remark}
By interpolation inequalities and iteration, we can replace the term $r^{-1}(|Du-[Du]_{B_r}(t)|^{q_0})_{Q_{r}}^{1/q_0}$ in \eqref{eq9.19} by $r^{-2}(|u-[u]_{B_r}(t)|^{q_0})_{Q_{r}}^{1/q_0}$.
\end{remark}

In the next lemma we prove a mean oscillation estimate of $D\omega$.

\begin{lemma}
            \label{non-vorticity-oss.est}
Let $q_1 \in (1, \infty), q_0 \in (1, q_1)$, $\delta \in (0,1)$, $R_0\in (0,1/4)$, $r \in (0, R_0)$, $\kappa \in (0, 1/4)$. Suppose that  \eqref{ellipticity} and Assumption \ref{assump1} ($\delta, 1$) hold. Suppose that $u\in W^{1,2}_{q_1}(Q_{r})^d$ is a strong solution to \eqref{eq7.41c} in $Q_{r}$. Then it holds that
\begin{align*}
(D\omega -(D\omega)_{Q_{\kappa r}})_{Q_{\kappa r}} & \leq N(d, \nu, q_0) \kappa^{-\frac{n+2}{q_0}}(|f|^{q_0})_{Q_{r}}^{1/q_0}  + N(d, q_0) \kappa (|Dg|^{q_0})_{Q_{r}}^{1/q_0} \\
& \quad \quad + N(n, \nu, q_0, q_1)\Big(\kappa^{-\frac{n+2}{q_0}} \delta^{1/q_0-1/q_1} + \kappa\Big) (|D^2u|^{q_1})_{Q_{r}}^{1/{q_1}},
\end{align*}
where $\omega = \nabla \times u$.
\end{lemma}
\begin{proof}
The proof is similar to that of Lemma \ref{div-vorticity-oss.est}. Let $(w, p_1)$ and $(v, p_2)$ be as in the proof of Lemma \ref{lem2.4c}.  In particular,  $(w, p_1)$ is a strong solution of
\[
w_t - \bar{a}_{ij}(t) D_{ij} w + \nabla p_1 = I_{Q_{r}} [f + (a_{ij} - \bar{a}_{ij}(t)) D_{ij} u], \quad \Div w = \phi_{r} (g - [g]_{B_{r}}(t))
\]
in $(-r^2, 0) \times \bR^d$ with zero initial condition on $\{t = -r^2\}$. Moreover, $(v, p_2) = (u-w, p- p_1)$ is a strong solution of
\[
v_t - \bar{a}_{ij}(t) D_{ij} v + \nabla p_2 = 0, \quad \Div v = [g]_{B_{r}}(t)
\]
in $Q_{2r/3}$. Let $\omega_1 = \nabla \times w$ and $\omega_2 = \nabla \times v$ and we see that $\omega = \omega_1 + \omega_2$. Moreover, we can deduce from \eqref{eq8.37c} that
\begin{align}
(|D\omega_1|^{q_0})_{Q_{r}}^{1/q_0}  \leq (|D^2 w|^{q_0})_{Q_{r}}^{1/q_0} & \leq  N(d, q_0) (|Dg|^{q_0})_{Q_{r}}^{1/q_0}  \nonumber \\ \label{non-omega-1.est}
& \quad + N(d, \nu, q_0)\Big[  (|f|^{q_0})_{Q_{r}}^{1/q_0} + \delta^{1/q_0-1/q}(|D^2u|^{q_1})_{Q_{r}}^{1/q_1} \Big].
\end{align}
Also, by applying Lemma \ref{lem1.3} to $D\omega_2$ with a suitable scaling, we obtain
\begin{align*}
(D\omega_2 - (D\omega_2)_{Q_{\kappa r}})_{Q_{\kappa r}} & \leq N\kappa r [[D\omega_2]]_{C^{1/2, 1}(Q_{r/3})} \leq N(d, \nu, q_0) \kappa (|D\omega_2|^{q_0})_{Q_{2r/3}}^{1/q_0}  \\
& \leq  N(d, \nu, q_0) \kappa \Big[ (|D\omega|^{q_0})_{Q_{r}}^{1/q_0} + (|D\omega_1|^{q_0})_{Q_{r}}^{1/q_0}\Big].
\end{align*}
Then, by combining this last estimate with \eqref{non-omega-1.est} and the fact that $\delta \in (0,1)$, we infer that
\begin{align}\nonumber
(D\omega_2 - (D\omega_2)_{Q_{\kappa r}})_{Q_{\kappa r}}& \leq N(d, q_0)\kappa (|Dg|^{q_0})_{Q_{r}}^{1/q_0}   \\  \label{non-omega-2.est}
& \quad +   N(d, \nu, q_0)\kappa  \Big[  (|f|^{q_0})_{Q_{r}}^{1/q_0} + (|D^2u|^{q_1})_{Q_{r}}^{1/q_1}\Big].
\end{align}
Now, by using the inequality
\[
\fint_{Q_{\kappa r}}|D\omega -(D\omega)_{Q_{\kappa r}}| \,dx\, dt \leq 2 \fint_{Q_{\kappa r}}|D\omega -c| \,dx\,dt
\]
with $c = (D\omega_2)_{Q_{\kappa r}}$, and then applying
the triangle inequality, and H\"{o}lder's inequality, we have
\begin{align*}
&\fint_{Q_{\kappa r}}|D\omega -(D\omega)_{Q_{\kappa r}}| \,dx\, dt   \leq 2 \fint_{Q_{\kappa r}} |D\omega - (D\omega_2)_{Q_{\kappa r}}| \,dx\, dt\\
& \leq 2  \fint_{Q_{\kappa r}} |D\omega_2 - (D\omega_2)_{Q_{\kappa r}}| \,dx\, dt  + N(d, q_0) \kappa^{-\frac{d+2}{q_0}}\left(\fint_{Q_{r}} |D\omega_1|^{q_0} \,dx\, dt\right)^{1/q_0}.
\end{align*}
This last estimate together with  \eqref{non-omega-1.est} and \eqref{non-omega-2.est} imply that
\begin{align*}
(D\omega -(D\omega)_{Q_{\kappa r}})_{Q_{\kappa r}} & \leq N(d, \nu, q_0) \kappa^{-\frac{d+2}{q_0}}(|f|^{q_0})_{Q_{r}}^{1/q_0} +   N(d, q_0)\kappa (|Dg|^{q_0})_{Q_{r}}^{1/q_0}   \\
& \quad \quad + N(d, \nu, q_0, q_1)\Big(\kappa^{-\frac{d+2}{q_0}} \delta^{1/q_0-1/q}+ \kappa\Big) (|D^2u|^{q_1})_{Q_{r}}^{1/q_1}.
\end{align*}
The proof is then complete.
\end{proof}

Our next lemma give the key estimates of $D\omega $ and $D^2u$ in the mixed norm.
\begin{lemma}
Let $R\in (0,R_0)$, $\delta \in (0,1)$, $\kappa \in (0, 1/4)$, $s, q \in (1, \infty)$,  $q_1 \in (1,\min\{s,q\})$, and $q_0 \in (1, q_1)$.
Assume that \eqref{ellipticity} and Assumption \ref{assump1} ($\delta, 1$) hold.
Suppose that $u\in W^{1,2}_{s,q}(Q_R)^d$ is a strong solution to \eqref{eq7.41c} in $Q_R$, and $\omega = \nabla\times u$. Then
we have
\begin{align}
&\norm{D\omega}_{L_{s,q}(Q_{2R/3})} \leq N \kappa^{-\frac{d+2}{q_0}}\norm{f}_{L_{s,q}(Q_{3R/4})} + N\kappa \norm{Dg}_{L_{s,q}(Q_{3R/4})}  \nonumber\\
                            \label{eq1.39n}
&\quad+N\Big(\kappa^{-\frac{d+2}{q_0}} \delta^{1/q_0-1/q_1} + \kappa\Big) \norm{D^2u}_{L_{s,q}(Q_{3R/4})}+ N R^{2/s+d/q}\kappa^{-d-2}(|D\omega|)_{Q_{3R/4}}
\end{align}
and
\begin{align}
&\norm{D^2u}_{L_{s,q}(Q_{R/2})} \leq N \kappa^{-d-2}\norm{f}_{L_{s,q}(Q_{R})} +N \kappa^{-d-2} \norm{Dg}_{L_{s,q}(Q_{3R/4})}   \nonumber\\
                            \label{eq1.40n}
&\quad+N\Big(\kappa^{-d-2} \delta^{1/q_0-1/q_1} + \kappa\Big) \norm{D^2u}_{L_{s,q}(Q_{R})}+
N \kappa^{-d-2} R^{-1}\norm{Du}_{L_{s,q}(Q_{R})}.
\end{align}
\end{lemma}
\begin{proof} As in the proof of Lemma \ref{lem4.10}, we discuss two cases.
\\ \noindent
{\em Case 1: $r \in (0, R/12)$.} It follows from Lemma \ref{non-vorticity-oss.est} that for all $z_0 \in Q_{2R/3}$,
\begin{align*}
&(D\omega -(D\omega)_{Q_{\kappa r}(z_0)})_{Q_{\kappa r}(z_0)} \leq
N(d, \nu, q_0) \kappa^{-\frac{d+2}{q_0}} (|f|^{q_0})^{1/q_0}_{Q_{r}(z_0)} \\
& + N(d, q_0) \kappa (|Dg|^{q_0})^{1/q_0}_{Q_{r}(z_0)}
+ N(d, \nu, q_0, q_1)\Big(\kappa^{-\frac{d+2}{q_0}} \delta^{1/q_0-1/q_1}+ \kappa\Big)(|D^2u|^{q_1})^{1/q_1}_{Q_{r}(z_0)}.
\end{align*}
Observe that because $r < R/12$, we have $Q_{r}(z_0) \subset Q_{3R/4}$. Therefore,
\[
\begin{split}
&  (|Dg|^{q_0})^{1/q_0}_{Q_{r}(z_0)} \leq \mathcal{M}(I_{Q_{_{3R/4}}}|Dg|^{q_0})^{1/q_0}(z_0), \\
& (|f|^{q_0})^{1/q_0}_{Q_{r}(z_0)} \leq \mathcal{M}(I_{Q_{_{3R/4}}}|f|^{q_0})^{1/q_0}(z_0),  \quad  \text{and} \\
& (|D^2u|^{q_1})^{1/q_1}_{Q_{r}(z_0)}  \le \mathcal{M}(I_{Q_{_{3R/4}}}|D^2u|^{q_1})^{1/q_1}(z_0),
\end{split}
\]
where $\mathcal{M}$ is the Hardy-Littlewood maximal function. These estimates imply that
\begin{align*}
&(D\omega -(D\omega)_{Q_{\kappa r}(z_0)})_{Q_{\kappa r}(z_0)}
 \leq  N \kappa^{-\frac{d+2}{q_0}} \mathcal{M}(I_{Q_{_{3R/4}}}|f|^{q_0})^{1/q_0}(z_0)\\
&\quad\quad +N(d, q_0) \kappa \mathcal{M}(I_{Q_{_{3R/4}}}|Dg|^{q_0})^{1/q_0}(z_0)
+ N\Big(\kappa^{-\frac{d+2}{q_0}} \delta^{1/q_0-1/q_1}+ \kappa\Big)\mathcal{M}(I_{Q_{_{3R/4}}}|D^2u|^{q_1})^{1/q_1}(z_0).
\end{align*}
\noindent
{\em Case 2: $r\in [R/12, R/(12\kappa))$.} In this case, we simply estimate
$$
(D\omega -(D\omega)_{Q_{\kappa r}(z_0)})_{Q_{\kappa r}(z_0)}
\le 2(|D\omega|)_{Q_{\kappa r}(z_0)}
\le N\kappa^{-d-2}(|D\omega|)_{Q_{3R/4}}.
$$
Now, we take $\cX=Q_{2R/3}$ and define the dyadic sharp function $(D\omega)^\#_{\text{dy}}$ of $D\omega$ in $\cX$. From the above two cases, we conclude that for any $z_0\in \cX$,
\begin{align*}
&(D\omega)_{\text{dy}}^{\#}(z_0)    \leq N(d, \nu, q_0) \kappa^{-\frac{d+2}{q_0}} \mathcal{M}(I_{Q_{_{3R/4}}}|f|^{q_0})^{1/q_0}(z_0) +  N(d, q_0) \kappa \mathcal{M}(I_{Q_{_{3R/4}}}|Dg|^{q_0})^{1/q_0}(z_0) \\
& \quad  + N(d, \nu, q_0, q_1)\Big(\kappa^{-\frac{d+2}{q_0}} \delta^{ 1/q_0-1/q_1}+ \kappa\Big) \mathcal{M}(I_{Q_{_{3R/4}}}|D^2u|^{q_1})^{1/q_1} (z_0) + N\kappa^{-d-2}(|D\omega|)_{Q_{3R/4}}.
\end{align*}
Recalling that $1<q_0<q_1<\min\{s,q\}$, by Lemma \ref{mixed-norm-lemma} and the Hardy-Littlewood maximum function theorem in mixed-norm spaces (see, for instance, \cite[Corollary 2.6]{DK16}),
\begin{align*}
&\norm{D\omega}_{L_{s,q}(Q_{2R/3})}
\leq N\Big[ \norm{(D\omega)_{\text{dy}}^{\#}}_{L_{s,q}
(Q_{2R/3})} + R^{2/s+d/q}(|D\omega|)_{Q_{3R/4}} \Big]\\
& \leq  N \kappa^{-\frac{d+2}{q_0}} \norm{\mathcal{M}(I_{Q_{_{3R/4}}}|f|^{q_0})^{1/q_0}}_{L_{s,q}(\bR^{d+1})}  +  N(d, q_0) \kappa \norm{\mathcal{M}(I_{Q_{_{3R/4}}}|Dg|^{q_0})^{1/q_0}}_{L_{s,q}(\bR^{d+1})}\\
&\quad+N\Big(\kappa^{-\frac{d+2}{q_0}} \delta^{1/q_0-1/q_1}+ \kappa\Big) \norm{\mathcal{M}(I_{Q_{_{3R/4}}}|D^2u|^{q_1})^{1/q_1}}_{L_{s,q}(\bR^{d+1})} + NR^{2/s+d/q}\kappa^{-d-2}(|D\omega|)_{Q_{3R/4}}\\
&\le N \Big[ \kappa^{-\frac{d+2}{q_0}}\norm{f}_{L_{s,q}(Q_{3R/4})} + \kappa \norm{Dg}_{L_{s,q}(Q_{3R/4})} \\
& \quad +\Big(\kappa^{-\frac{d+2}{q_0}} \delta^{1/q_0-1/q_1} + \kappa\Big) \norm{D^2u}_{L_{s,q}(Q_{3R/4})}+  R^{2/s+d/q}\kappa^{-d-2}(|D\omega|)_{Q_{3R/4}}\Big],
\end{align*}
which gives \eqref{eq1.39n}.

Next we show \eqref{eq1.40n}. Since $\Div u =g$, as in the proof of Lemma \ref{lem1.2}, we have
\begin{equation}
                            \label{eq1.52n}
\|D^2u\|_{L_{s,q}(Q_{R/2})}
\le N\|D\omega\|_{L_{s,q}(Q_{2R/3})} + N\|Dg\|_{L_{s,q}(Q_{2R/3})}  +N R^{-1}\|Du\|_{L_{s,q}(Q_{2R/3})}.
\end{equation}
We also use H\"older's inequality and Lemma \ref{lem2.4c} with a covering argument to estimate the last term in \eqref{eq1.39n} by
\begin{align}
&R^{2/s+d/q}\kappa^{-d-2}(|D\omega|)_{Q_{3R/4}}
\le NR^{2/s+d/q}\kappa^{-d-2}(|D^2 u|^{q_0})^{1/q_0}_{Q_{3R/4}}\nonumber\\
&\le NR^{2/s+d/q}\kappa^{-d-2}\Big[(|f|^{q_0})_{Q_{R}}^{1/q_0} + (|Dg|^{q_0})_{Q_{R}}^{1/q_0}
+R^{-1}(|Du|^{q_0})_{Q_{R}}^{1/q_0}+\delta^{1/q_0 - 1/q_1} (|D^2 u|^{q_1})_{Q_{R}}^{1/q_1}\Big]\nonumber\\
                    \label{eq2.01n}
&\le N\kappa^{-d-2}\Big(\|f\|_{L_{s,q}(Q_{R})} + \|Dg\|_{L_{s,q}(Q_{R})}
+R^{-1}\| Du\|_{L_{s,q}(Q_{R})}+\delta^{1/q_0 - 1/q_1}\|D^2 u\|_{L_{s,q}(Q_{R})}\Big).
\end{align}
Combining \eqref{eq1.52n}, \eqref{eq1.39n}, and \eqref{eq2.01n}, we reach \eqref{eq1.40n}. The lemma is proved.
\end{proof}

Now we are ready to give

\begin{proof}[Proof of Theorem \ref{thm2.3b}]
As in the proof of Lemma \ref{thm2.3}, for $k=1,2,\ldots$, we denote $Q^k=(-(1-2^{-k})^2,0)\times B_{1-2^{-k}}$. Let $k_0$ be the smallest positive integer such that $2^{-k_0-1}\le R_0$. For $k\ge k_0$, we apply \eqref{eq1.40n} with $R=2^{-k-1}$ and a covering argument to get
\begin{align}
&\norm{D^2u}_{L_{s,q}(Q^{k})} \leq N \kappa^{-d-2}\norm{f}_{L_{s,q}(Q^{k+1})} + N \kappa^{-d-2}\norm{Dg}_{L_{s,q}(Q^{k+1})} \nonumber\\
                            \label{eq2.15n}
&\quad+N\Big(\kappa^{-d-2} \delta^{1/q_0-1/q_1} + \kappa\Big) \norm{D^2u}_{L_{s,q}(Q^{k+1})}+
N \kappa^{-d-2}2^{k}\norm{Du}_{L_{s,q}(Q^{k+1})}.
\end{align}
From \eqref{eq2.15n} and interpolation inequalities, we get
\begin{align}
&\norm{D^2u}_{L_{s,q}(Q^{k})} \leq N \kappa^{-d-2}\norm{f}_{L_{s,q}(Q^{k+1})} + N \kappa^{-d-2}\|Dg\|_{L_{s,q}(Q^{k+1})} \nonumber\\
                            \label{eq3.21}
&\quad+N\Big(\kappa^{-d-2} \delta^{1/q_0-1/q_1} + \kappa\Big) \norm{D^2u}_{L_{s,q}(Q^{k+1})}+
N \kappa^{-2d-5}2^{2k}\norm{u}_{L_{s,q}(Q^{k+1})}.
\end{align}
Note that the constants $N$ above are independent of $k$.
We then take $\kappa$ sufficiently small and then $\delta$ sufficiently small so that
$$
N\Big(\kappa^{-d-2} \delta^{1/q_0-1/q_1} + \kappa\Big)\le 1/5.
$$
Finally, we multiply both sides of \eqref{eq3.21} by $5^{-k}$ and sum in $k=k_0,k_0+1,\ldots$ to get the desired estimate. The theorem is proved.
\end{proof}

\section{Regularity for Navier-Stokes equations} \label{NS-se}

To prove Theorem \ref{NS-reg.thm}, let us recall several well-known results needed for the proof. The first result is the classical regularity criterion for Leray-Hopf weak solutions of the Navier-Stokes equations established in \cite{Serrin}.
\begin{theorem} \label{Serrin-thm}
For each $\rho>0$, let $u$ be a Leray-Hopf weak solution of the Navier-Stokes equations \eqref{NS.eqn} in $Q_\rho$ which satisfies
\[\sup_{t \in (-\rho^2, 0)} \int_{B_\rho}|u(t,x)|^2 \,dx + \int_{Q_\rho} |\nabla u(t,x)|^2 \,dx\,dt < \infty,\]
and $\norm{u}_{L_{s,q}(Q_\rho)} <\infty$ with some $s, q \in (1, \infty)$ such that
\[
d/q + 2/s < 1.
\]
Then, $u$ is smooth in $Q_\rho$.
\end{theorem}
The following classical parabolic Sobolev embedding theorem will be used iteratively in the proof.
\begin{lemma} \label{Sob-imbdding} For each $m >1$, let $q = \frac{m(d+2)}{d}$. Then, for each $\rho>0$, there exists a constant $N = N(d, m, \rho)>0$ such that
\[
\norm{f}_{L_{q}(Q_\rho)} \leq N \sup_{t\in (-\rho^2, 0)} \norm{f(t, \cdot)}_{L_2(B_\rho)} + N\norm{\nabla f}_{L_m(Q_\rho)}.
\]
\end{lemma}
Now, we are ready to prove Theorem \ref{NS-reg.thm}

\begin{proof}[Proof of Theorem \ref{NS-reg.thm}]
For the reader's convenience, we recall that $(d_{ij})_{d \times d}$ is the skew-symmetric matrix which satisfies the equation
\begin{equation} \label{d-matrix.eqn}
\left \{
\begin{array}{cccl}
\Delta d_{ij} & = & D_j u_i - D_i u_j & \quad B_1, \\
d_{ij} & = &0 & \quad \partial B_1.
\end{array} \right.
\end{equation}
Then, by the energy estimate for the equation \eqref{d-matrix.eqn}, we see that
\begin{equation} \label{L-p-d.est}
\sup_{t\in (-1,0)}  \int_{B_{1}} |D d_{ij}(t,x)|^2 \,dx  \leq 4 \sup_{t\in (-1,0)}  \int_{B_{1}} |u (t,x)|^{2} \,dx < \infty, \quad \forall \ i, j = 1,2,\ldots, d.
\end{equation}
Let us now denote $h = (h_1, h_2,\ldots, h_d)$ by
\[
 h_j (t,x) = -u_j (t,x) - \sum_{i=1}^d D_i d_{ij}(t,x) , \quad (t,x) \in Q_1, \ j = 1, 2,\ldots, d.
\]
Observe that $\Div h (t, \cdot)  =0$ in the sense of distributions in $B_1$ for a.e. $t \in (-1,0)$.  Then, we can write the nonlinear term in \eqref{NS.eqn} as
\[
u \cdot \nabla u_k = \sum_{j=1}^d u_j D_j u_k  =
-\sum_{i=1}^d D_i d_{ij} D_j u_k -\sum_{j=1}^d D_j [h_j u_k ].
\]
As the matrix $(d_{ij})_{d\times d}$ is skew-symmetric, we see that
\[
\sum_{i, j =1}^d \int D_i d_{ij} D_j u_k \varphi \,dx = - \int  d_{ij} D_j u_k D_i \varphi \, dx ,\quad \forall \varphi \in C_0^\infty(B_1).
\]
Consequently, $u$ is also a weak solution of the Stokes system
\begin{equation} \label{NS-proof.eqn}
u_t - D_i[(I_d + d_{ij})D_j u] + \nabla p =  \Div f \quad \text{in} \quad Q_1,
\end{equation}
where $I_d$ is the $d\times d$ identity matrix, and $f_{jk} = h_j u_k$.

Next, for each $k  \in \mathbb{N}$, we define the following sequences
\[
s_0 =2, \quad s_{k+1} = s_k \frac{d+2}{d}, \quad r_{k} = \frac 1 2+\frac 1 {2^{k+1}}.
\]
Let $k_0 \in \mathbb{N}$ be sufficiently large such that
\begin{equation} \label{choice-k-0}
\frac{d}{s_{k_0}} + \frac{2}{s_{k_0}} < 1,
\end{equation}
and let
\begin{equation} \label{choice-ep}
\epsilon = \min\Big \{ \delta(d, 1, s_{k}, s_{k},\alpha_0), \ k=1,2,\ldots,k_0\Big \},
\end{equation}
where $\delta(d, 1, s_{k}, s_{k},\alpha_0)$ is defined in Theorem \ref{thm2.3}.
Assume that \eqref{epsilon-criteria} holds and we will prove Theorem \ref{NS-reg.thm} with this choice of $\epsilon$. To this end, we first observe that as $(s_k)_{k \in \mathbb{N}}$ is an increasing sequence
\begin{equation} \label{check-alpha}
\alpha_0 > \frac{2(d+2)}{d+4} = \frac{s_1}{s_1-1} \geq \frac{s_k}{s_k -1}, \quad \forall \ k \in \mathbb{N}.
\end{equation}
Now, from \eqref{d-matrix.eqn}, we see that $h_j(t, \cdot)$ is a harmonic function for a.e. $t \in (-1,0)$, i.e.,
\[
\Delta h_j (t, \cdot) = 0 \quad \text{in} \quad B_1, \quad \forall \ j = 1, 2,\ldots, d.
\]
Therefore, it follows this and the estimate \eqref{L-p-d.est} that for any $\rho \in (0, 1)$, we have
\begin{align}  \nonumber
\norm{h}_{L_\infty(Q_{\rho})} & = \sup_{t\in (-1,0)}\norm{h(t, \cdot)}_{L_\infty(B_{\rho})} \leq N(d, \rho) \sup_{t\in (-1,0)} \left ( \fint_{B_{1}} |h(t, x)|^2 \,dx \right)^{1/2}  \\ \nonumber
& \leq N(d, \rho) \sup_{t\in (-1,0)}  \left[\left ( \fint_{B_{1}} |u(t, x)|^2 \,dx \right)^{1/2} + \sum_{i,j=1}^d\left ( \fint_{B_{1}} |Dd_{ij}(t, x)|^2 \,dx \right)^{1/2} \right] \\ \label{h-L-infty.est}
& \leq C(d, \rho) \sup_{t \in (-1, 0)} \left ( \fint_{B_{1}} |u(t, x)|^2 \,dx \right)^{1/2} < \infty.
\end{align}

Let us also denote
\[
\norm{u}_{V_{s_k}(Q_{r_k})} = \sup_{t \in (-r_k^2, 0)}\left( \int_{B_{r_k}} |u(t,x)|^2\, dx \right)^{1/2} + \left(\int_{Q_{r_k}} |\nabla u(t,x)|^{s_k}\, dx\,dt \right)^{1/s_k}.
\]
Observe that as $\norm{u}_{V_{s_0}(Q_{r_0})} < \infty$, it follows from Lemma \ref{Sob-imbdding}, $u \in L_{s_1}(Q_{r_0})$. From this, \eqref{choice-ep}, \eqref{check-alpha}, and \eqref{h-L-infty.est},  we can apply Theorem \ref{thm2.3} to the equation \eqref{NS-proof.eqn} to obtain
\[
\norm{\nabla u}_{L_{s_1}(Q_{r_1})} \leq N_1 \Big[ \norm{u}_{L_{s_1}(Q_{r_0})} + \norm{f}_{L_{s_1}(Q_{r_0})} \Big] \leq N \norm{u}_{L_{s_1}(Q_{r_0})} < \infty.
\]
Consequently, we see that $\norm{u}_{V_{s_1}(Q_{r_1})} <\infty$, therefore it follows from Lemma \ref{Sob-imbdding} again that  $u \in L_{s_2}(Q_{r_1}) < \infty$. Hence, by applying Theorem \ref{NS-reg.thm} again to \eqref{NS-proof.eqn} we infer that
\[
\norm{\nabla u}_{L_{s_2}(Q_{r_2})} \leq N_2 \norm{u}_{L_{s_2}(Q_{r_1})} < \infty.
\]
Repeating this procedure, we then conclude that
\[
\norm{\nabla u}_{L_{s_{k+1}}(Q_{r_{k+1}})} \leq N_{k+1} \norm{u}_{L_{s_{k+1}}(Q_{r_k})} < \infty, \quad \forall \ k \in \{0,1, 2, \ldots, k_0-1\}.
\]
From this estimate, Theorem \ref{Serrin-thm}, and  the choice of $k_0$ in \eqref{choice-k-0}, we see that the conclusion of Theorem \ref{NS-reg.thm} follows. The proof is then complete.
\end{proof}
Finally, we conclude our paper with the proof of Corollary \ref{coro}.

\begin{proof}[Proof of Corollary \ref{coro}] Note that when $s<\infty$, the boundedness of $L_s((-1,0);L_{q}^w(B_1))$ norm of $u$ implies the smallness of the same norm of $u$ in small cylinders. Therefore, $(i)$ follows from $(ii)$. Moreover, by the well-known embedding:
$$
L_q^w \hookrightarrow \mathscr{M}_{q_1,\lambda},\quad \text{when}\quad q_1\in [1,q)\,\,\text{and}\,\, \lambda=d(1-q_1/q),
$$
it suffices for us to prove $(iii)$.

By the Calder\'on-Zygmund estimate in Morrey spaces for the Laplace equation (see, for instance, \cite{Ra00}), we have
$$
\|Dd_{ij}(t, \cdot) \|_{\mathscr{M}_{q,\beta}(B_1)}\le N(d,q,\beta)\|u(t, \cdot)\|_{\mathscr{M}_{q,\beta}(B_1)} \quad \text{for a.e.} \ t \in (-1,0),
$$
which implies that
\begin{equation}
                                        \label{eq9.00}
\|Dd_{ij}\|_{\mathscr{M}_{s,\alpha}((-1,0); \mathscr{M}_{q,\beta}(B_1))}\le N(d,s,q,\alpha,\beta)\|u\|_{\mathscr{M}_{s,\alpha}((-1,0);
\mathscr{M}_{q,\beta}(B_1))}.
\end{equation}
On the other hand, for $\alpha_0 \in (\frac{2(d+2)}{d+4}, \min\{s, dq/(d-q)_+\})$ with $(d-q)_+ = \min\{0, d-q\}$, by using \eqref{eq5.13}, the Sobolev-Poincar\'{e} inequality, and H\"{older}'s inequality, we see that for any $z_0\in Q_{2/3}$ and $\rho\in (0,1/3)$,
\begin{align} \nonumber
& \fint_{Q_\rho(z_0)} |d_{ij}(t,x) - [d_{ij}]_{B_\rho(x_0)}(t)|^{\alpha_0}\, dx\, dt  \\ \nonumber
& \leq  N(d,\alpha_0) \rho^{\alpha_0} \left [\fint_{t_0 -\rho^2}^{t_0} \left( \fint_{B_\rho(x_0)} |D d_{ij}(t,x)|^{q}\, dx \right)^{s/q}\, dt \right]^{\frac{\alpha_0}{s}} \\ \label{NS-reg-2.est}
& = N(d,\alpha_0) \|Dd_{ij}\|^{\alpha_0}_{\mathscr{M}_{s,\alpha}((-1,0);
\mathscr{M}_{q,\beta}(B_1))}.
\end{align}
By combining \eqref{eq9.00} and \eqref{NS-reg-2.est}, we can apply Theorem \ref{NS-reg.thm} to conclude the proof.
\end{proof}

\section*{Acknowledgement}

The authors would like to thank Doyoon Kim for helpful discussions on the subject. 


\end{document}